\documentclass[a4paper,draft]{article}
\usepackage{amsthm,amssymb,amsmath,enumitem} 
\usepackage{graphicx,mathtools}

\usepackage[T1]{fontenc}
\usepackage[utf8]{inputenc}
\usepackage{authblk}

\newcommand{\N}{\ensuremath{\mathbb{N}}}

\newcommand{\Z}{\ensuremath{\mathbb{Z}}}

\newcommand{\Vcal}{\ensuremath{\mathcal{V}}}
\newcommand{\Wcal}{\ensuremath{\mathcal{W}}}

\newcommand{\Gcal}{\ensuremath{\mathcal{G}}}

\newcommand{\G}{\Gamma}

\newcommand{\free}{\ast}

\newcommand{\inv}{\ensuremath{^{-1}}}
\newcommand{\sub}{\subseteq}
\newcommand{\sm}{\smallsetminus}
\newcommand{\id}{\ensuremath{\mathrm{id}}}

\newcommand{\fwd}{\overset{\scriptscriptstyle\rightarrow}}
\newcommand{\bwd}{\overset{\scriptscriptstyle\leftarrow}}

\theoremstyle{definition}

\theoremstyle{plain}

\newtheorem{theorem}{Theorem}[section]
\newtheorem{thm}[theorem]{Theorem}
\newtheorem{conj}[theorem]{Conjecture}
\newtheorem{coro}[theorem]{Corollary}
\newtheorem{lemma}[theorem]{Lemma}
\newtheorem{prop}[theorem]{Proposition}

\theoremstyle{remark}
\newtheorem{remark}[theorem]{Remark}

\theoremstyle{remark}

\def\td{tree-decom\-po\-si\-tion}
\def\ta{tree amalgamation}
\def\qt{quasi-tran\-si\-tive}
\def\qg{quasi-geo\-de\-sic}

\newcommand{\comment}[1]{}

\begin{document}

\title{A Stallings' type theorem\\for \qt\ graphs}
\author{Matthias Hamann\thanks{Supported by the European Research Council under the European Union's Seventh Framework Programme (FP7/2007-2013) / ERC grant agreement n$^\circ$ 617747 and through the Heisenberg-Programme of the Deutsche Forschungsgemeinschaft (DFG Grant HA 8257/1-1).} \\
	Alfr\'{e}d R\'{e}nyi Institute of Mathematics \\ Budapest, Hungary
	\and 
	Florian Lehner \thanks{Supported by the Austrian Science Fund (FWF), grant J 3850-N32}\\
	Mathematics Institute, University of Warwick \\ Coventry, UK
	\and
	Babak Miraftab \\
	Department of Mathematics, University of Hamburg \\ Hamburg, Germany
	\and
	Tim R\"uhmann \\
	Department of Mathematics, University of Hamburg \\ Hamburg, Germany
	}

\date{\today}

\maketitle
\begin{abstract}
We consider infinite connected \qt\ locally finite graphs and show that every such graph with more than one end is a \ta\ of two other such graphs.
This can be seen as a graph-theoretical version of Stallings' splitting theorem for multi-ended finitely generated groups and indeed it implies this theorem.
It will also lead to a characterisation of accessible graphs.
We obtain applications of our results for hyperbolic graphs, planar graphs and graphs without any thick end.
The application for planar graphs answers a question of Mohar in the affirmative.
\end{abstract}

\section{Introduction}
\label{Intro}
Stallings~\cite{Stallings} proved  that finitely generated groups with more than one end are either a free product with amalgamation over a finite subgroup or an HNN-extension over a finite subgroup.
The main aim of this paper is to obtain an analogue of Stallings' theorem for quasi-transitive graphs.
The obvious obstacle for this is that free products with amalgamations and HNN-extensions are group theoretical concepts.
So in order to obtain a graph-theoretical analogue, we first need to  find a graph-theoretical analogue of free products with amalgamations and HNN-extensions.
The proposed notation by Mohar~\cite{Mohar06} are \ta s and indeed we will prove the following theorem.
(We refer to Section~\ref{Sec_Splitting} for the definition of \ta s.)

\begin{thm}\label{stallingsgraphIntro}
Every connected \qt\ locally finite graph with more than one end is a non-trivial tree amalgamation of finite adhesion of two connected \qt\ locally finite graphs.
\end{thm}

On the other side, we can ask if we start with finite or one-ended connected \qt\ locally finite graphs and do iterated \ta s of finite adhesion, what class of graphs do we end up with?
In the case of finitely generated groups, the answer is the class of accessible groups (by definition).
Thomassen and Woess~\cite{ThomassenWoess} defined edge-accessibility for graphs\footnote{They call it accessible instead of edge-accessible, but we will reserve the notion of accessibility to a direct translation of accessibility of groups, see Section~\ref{Sec_Accessible}.}: a \qt\ locally finite graph is \emph{edge-accessible} if there is some $n\in\N$ such that every two ends can be separated by at most $n$ edges.
They showed in~\cite{ThomassenWoess} that a finitely generated group is accessible if and only if each of its locally finite Cayley graphs is edge-accessible.
We will show that \ta s and edge-accessibility fit well together in that we prove that the above described class of graphs we obtain is the class of edge-accessible connected \qt\ locally finite graphs.

In 1988, Mohar~\cite{Mohar06} asked whether \ta s are powerful enough to yield a classification of infinitely-ended transitive planar graphs in terms of finite and one-ended infinite planar transitive
graphs.
Our theorems enable us to answer his question in the affirmative for \qt\ graphs because Dunwoody~\cite{D-PlanarGraphsAndCovers} proved that they are edge-accessible, see Section~\ref{Sec_planar}.

Additionally, we obtain as a corollary Stallings' theorem, see Section~\ref{Sec_Stallings}, and a new characterisation of \qt\ locally finite graphs that are quasi-isometric to trees, see Section~\ref{Sec_thin}.
In Section~\ref{Sec_hyperbolic} we apply our theorems to hyperbolic graphs and show that a \qt\ locally finite graph is hyperbolic if and only if it is obtained by iterated \ta s starting with finite or one-ended hyperbolic \qt\ locally finite graphs.

Our main tool to prove Theorem~\ref{stallingsgraphIntro} are canonical \td s.
While some proofs of Stallings' theorem are build on edge separators and their structure trees, see e.\,g.\ Dunwoody~\cite{CuttingUpGraphs}, it turns out that \td s and vertex separators fit well together with \ta s.
Due to the similar natures of structure trees and \td s, it is not surprising that some results that we prove here (Propositions~\ref{BagDegree} and~\ref{capture thick ends_new}) have also been proved for structure trees, see e.\,g.\ Thomassen and Woess~\cite{ThomassenWoess} and M\"oller~\cite{RoggiEndsI, RoggiEndsII}.

\section{Preliminaries}
\label{prelim}
We follow the general notations of \cite{DiestelBook10noEE} unless stated otherwise. 
In the following we will  state the most important definitions for convenience. 

Let~$G=(V(G),E(G))$ be a graph.
A \emph{geodesic} is a shortest path between two vertices.
A \emph{ray} is a one-way infinite path, the infinite subpaths of a ray are its \emph{tails}.
Two rays are \emph{equivalent} if there exists no finite vertex set separating them eventually, i.\,e. two rays are equivalent if they have tails contained in the same component of $G-S$ for every finite set~$S$ of vertices.
The equivalence classes of rays in a graph are its \emph{ends}.
The \emph{degree} of an end is the maximum number of disjoint rays in that end, if it exists.  
If that maximum does not exist, we say that this end has \emph{infinite degree} and call it \emph{thick}.
An end with finite degree is called \emph{thin}. 
An end~$\omega$ is \emph{captured} by a set $X$ of vertices if every ray of~$\omega$ has infinite intersection with~$X$ and it \emph{lives} in~$X$ if every ray of~$\omega$ has a tail in~$X$.

Let $X\sub V(G)$.
Let $G'$ be the graph with vertex set $(V(G)\sm X)\cup\{v_X\}$, where $v_X$ is a new vertex, and edges between $u,v\in V(G)\sm X$ if and only if $uv\in E(G)$ and $v_x$ is adjacent to precisely those vertices $y\in V(G)\sm X$ that have a neighbour in~$X$.
We call $G'$ the \emph{contraction} of~$X$ in~$G$ and we say that we \emph{contracted}~$X$.
Since edges are just vertex sets of size~$2$, the definition carries over to edges.

Let $\G$ be a group acting on~$G$ and let $X\sub V(G)$.
The (\emph{setwise}) \emph{stabilizer} of~$X$ with respect to~$\Gamma$ is the set  
\[
\Gamma_X := \{g\in \G\mid g(x)\in X\text{ for all } x\in X \}.
\]
An \emph{orbit} of~$\G$ (or a \emph{$\G$-orbit}) is a set $\{g(x)\mid g\in\G\}$ for some $x\in V(G)$.
We say $\G$ acts \emph{transitively} on~$G$ if $V(G)$ is one $\Gamma$-orbit and $\G$ acts \emph{\qt ly} on~$G$ if $V(G)$ consists of finitely many $\Gamma$-orbits.


\section{Canonical \td s}\label{TD}

In this section we will look at our main tool for our proofs: \td s.
A \emph{\td} of a graph~$G$ is a pair $(T,\Vcal)$ where $T$ is a tree and $\Vcal=(V_t)_{t \in V(T)}$ is a family of vertex sets of~$G$ such that the following holds: 
\begin{enumerate}[label=(T\arabic*)]
\item $V(G) = \bigcup_{t \in V(T)} V_t$.
\item For every edge $e\in E(G)$ there is a $t\in V(T)$ such that $V_t$ contains both vertices that are incident with~$e$.
\item $V_{t_1}\cap V_{t_2} \subseteq V_{t_3}$ whenever~$t_3$ lies on the $t_1-t_2$ path in~$T$. 
\end{enumerate}
The sets~$V_t$ are called the \emph{parts} of $(T,\Vcal)$ and the vertices of the \emph{decomposition tree}~$T$ are its \emph{nodes}. 
The sets $V_{t_1} \cap V_{t_2}$ with $t_1t_2\in E(T)$ are the \emph{adhesion sets} of the \td.
We say that $(T,\Vcal)$ has \emph{finite adhesion} if all adhesion sets are finite.

\begin{remark}\label{rem_adhSeps}
Let $t_1t_2$ be an edge of the decomposition tree~$T$ of a \td\ $(T,\Vcal)$.
For $i=1,2$, let $T_i$ be the component of $T-t_1t_2$ that contains~$t_i$.
It follows from {\rm (T3)} that $V_{t_1} \cap V_{t_2}$ separates the vertices in $\bigcup_{t\in T_1}V_t$ from those in $\bigcup_{t\in T_2}V_t$.
\end{remark}

We say $(T,\Vcal)$ \emph{distinguishes} two ends $\omega_1$ and~$\omega_2$ if there is a finite adhesion set $V_{t_1} \cap V_{t_2}$ such that one end lives in $\bigcup_{t\in T_1}V_t$ and the other lives in $\bigcup_{t\in T_2}V_t$, where $T_i$ is the maximal subtree of $T-t_1t_2$ containing $t_i$.
It distinguishes them \emph{efficiently} if no vertex set in~$G$ of smaller size than $V_{t_1}\cap V_{t_2}$ separates them.
For $k\in\N$, two ends of~$G$ are \emph{$k$-distinguishable} if there is a set of $k$ vertices of $G$ that separates them.

Let $\Gamma$ be a group acting on~$G$.
If every $\gamma \in \Gamma$ maps parts of $(T,\Vcal)$ to parts and thereby induces an automorphism of~$T$ we say that $(T,\Vcal)$ is \emph{$\Gamma$-invariant}.

The following theorem by Carmesin et al.\ will be the main result we are building on.

\begin{thm}{\rm \cite{InfCanonicalTD}}\label{thm_canonicalTD}
Let $G$ be a locally finite graph, let $\Gamma$ be a group acting on~$G$ and let $k\in \N$.
Then there is a $\Gamma$-invariant \td\ of~$G$ of adhesion at most~$k$ that efficiently distinguishes all $k$-distinguishable ends.\qed
\end{thm}

\section{Basic \td s}

The aim of this section is first to modify the \td\ of Theorem~\ref{thm_canonicalTD} and then to prove some properties of the newly obtained \td, in particular, where the \td\ captures the ends of the graph.
Our first step in modifying the \td\ of Theorem~\ref{thm_canonicalTD} will be to make all adhesion sets connected while keeping the action of~$\Gamma$ on~$(T,\Vcal)$.

\begin{prop}\label{connected TD}
	Let~$\G$ be a group acting on a locally finite graph~$G$ and let $(T,\Vcal):=(T,(V_t)_{t\in V(T)})$ be a $\Gamma$-invariant \td\ of~$G$ of finite adhesion. 
	Then there is a $\Gamma$-invariant \td\ $(T,\Vcal'):=(T,(V'_t)_{t\in V(T)})$ of~$G$ such that every adhesion set of $(T,\Vcal')$ is finite and connected and such that $V_t\sub V'_t$ for every $t\in V(T)$.
\end{prop}

\begin{proof}
	Let~$u$ and~$v$ be two vertices of an adhesion set of $(T,\Vcal)$.
	Assume that~$\mathcal{P}_{uv}$ is the set of all geodesics between~$u$ and~$v$ and assume that~$V_{uv}$ is the set of all vertices of~$G$ that lie on the paths of~$\mathcal{P}_{uv}$.
	For a part $V_t$ let $V_t'$ be the union of $V_t$ with all sets $V_{uv}$ where $u$ and $v$ lie in an adhesion set contained in~$V_t$.
	Let $\Vcal':=\{V_t'\mid t\in V(T)\}$.
	We claim that $(T,\Vcal')$ is a \td.
	By construction it has the desired properties, i.\,e.\ every adhesion set is connected and $V_t\sub V_t'$ and, since $G$ is locally finite and since the adhesion sets of $(T,\Vcal)$ are finite, every adhesion set  of $(T,\Vcal')$ is finite.
	Since we made no choices when adding all possible geodesics to the adhesion sets, $\Gamma$ still acts on $(T,\Vcal')$.
	
	As every element of~$\Vcal'$ is a superset of some element of~$\Vcal$, we just have to verify (T3) to see that $(T,\Vcal')$ is a \td.
	To see this, let $x\in V'_{t_1}\cap V'_{t_2}$ for $t_1,t_2\in V(T)$ and let $t_3$ be on the $t_1$-$t_2$ path $s_1,\ldots,s_n$ in~$T$ with $s_1=t_1$ and $s_n=t_2$.
	If $x\in V_{t_1}\cap V_{t_2}$, then we have $x\in V_{t_3}\sub V'_{t_3}$ as $(T,\Vcal)$ is a \td.
	If $x\in (V'_{t_1}\sm V_{t_1})\cap V_{t_2}$, then it lies on a geodesic $P$ between two vertices $x_1,x_2$ of an adhesion set of $(T,\Vcal)$ in~$V_{t_1}$.
	Since every adhesion set $V_{s_i}\cap V_{s_{i+1}}$ separates $V_{s_1}$ from $V_{s_n}$ and since $x\in V_{t_2}$, the path $P$ must pass through $V_{s_i}\cap V_{s_{i+1}}$.
	Thus, either $P$ contains two vertices $u,v$ of $V_{s_i}\cap V_{s_{i+1}}$ such that $x$ lies on the $u$-$v$ subpath $P'$ of~$P$, or $x$ lies in $V_{s_i}\cap V_{s_{i+1}}$.
	In the first case, we added $P'$ to the adhesion set $V_{s_i}\cap V_{s_{i+1}}$ because $P'$ is a geodesic with its end vertices in $V_{s_i}\cap V_{s_{i+1}}$.
	Thus, in both cases $x$ lies in $V_{s_i}\cap V_{s_{i+1}}$ and thus in $V'_{t_3}$.
	If $x\in (V_{t_1}'\sm V_{t_1})\cap(V_{t_2}'\sm V_{t_2})$, let $t_4\in V(T)$ with $x\in V_{t_4}$.
	By the previous case, $x$ lies in~$V_t'$ for every $t$ on the $t_1$-$t_4$ or $t_2$-$t_4$ paths in~$T$.
	Since $T$ is a tree, these cover the path $s_1,\ldots, s_n$ and hence $x\in V_{t_3}'$.
	This proves that $(T,\Vcal')$ is a \td.
\end{proof}

We call a \td\ of a graph~$G$ \emph{connected} if all parts induce connected subgraphs of~$G$.

The step to make the adhesion sets connected is just an intermediate step for us as we aim for connected parts, i.\,e.\ we aim for connected \td s.
The next lemma ensures that in connected graphs all parts are connected if all adhesion sets are connected.

\begin{lemma}\label{nicetree}
	If all adhesion sets of a \td\ of a connected graph are connected, then the \td\ is connected.
\end{lemma}
\begin{proof}
Let $G$ be a graph and let $(T,\Vcal)$ be a \td\ of~$G$ all of whose adhesions sets are connected.
	Let $u$ and~$w$ be two vertices of~$V_t$ for some~$t\in V(T)$. 
	Since~$G$ is connected, there is a path~$P=p_1,\ldots, p_n$ with $p_1 =u$ and $p_n = w$. 
	We choose $P$ with as few vertices outside of~$V_t$ as possible.
	Let us suppose that~$P$ leaves~$V_t$.
	Let~$p_i \in V_t$ such that~$p_{i+1} \notin V_t$ and let~$p_{j}$ be the first vertex of~$P$ after~$p_i$ that lies in~$V_t$.
	As $p_n = w \in V_t$ we know that such a vertex always exists.
	Let $t'\in V(T)$ such that $p_{i+1}\in V_{t'}$.
	Then the adhesion set $V_t\cap V_s$, where $s$ is the neighbour of~$t$ on the $t$-$t'$ path in~$T$, separates $V_t$ from~$p_{i+1}$.
	Hence, the definition of a \td\ implies that $p_j$ must lie in~$V_t\cap V_s$, too.
	But then we can replace the subpath of~$P$ between $p_i$ and~$p_j$ by a path in $V_t\cap V_s$.
	The resulting walk contains a path between $u$ and~$w$ with fewer vertices outside of~$V_t$ than~$P$.
	This contradiction shows that all vertices of~$P$ lie in~$V_t$ and hence $G[V_t]$ is connected.
\end{proof}

Most of the time we do not need the full strength of Theorem~\ref{thm_canonicalTD} in that it suffices to consider $\Gamma$-invariant \td s with few $\Gamma$-orbits that still distinguish some ends.

Let~$\G$ be a group acting on a connected locally finite graph~$G$ with at least two ends. 
A $\Gamma$-invariant \td\ $(T,\mathcal{V})$ of~$G$ is a \emph{basic \td} (\emph{with respect to~$\G$}) if it has the following properties:
\begin{enumerate}[label=(\roman*)]
	\item $(T,\mathcal{V})$ distinguishes at least two ends.
	\item Every adhesion set of $(T,\mathcal{V})$ is finite.
	\item $\G$ acts on~$(T,\mathcal{V})$ with precisely one orbit on~$E(T)$.
\end{enumerate}
If it is clear from the context which group we consider, we just say that $(T,\mathcal{V})$ is a basic \td\ of~$G$.
It follows from Theorem~\ref{thm_canonicalTD} that basic \td s always exist:

\begin{coro}\label{nicetreedecomposition}
	Let~$\G$ be a group acting on a locally finite graph~$G$ with at least two ends. 
	Then there is a basic \td~$(T,\mathcal{V})$ for~$G$.
\end{coro}

\begin{proof}
	By Theorem~\ref{thm_canonicalTD}, we find a $\Gamma$-invariant \td\ $(T,\Vcal)$ of bounded adhesion that separates some ends.
	Let $tt'$ be an edge of~$T$ such that $V_t\cap V_{t'}$ separates some ends.
	Let $E_{tt'}$ be the orbit of~$tt'$, i.\,e.\ the set $\{g(tt')\mid g\in\Gamma\}$, and let $T'$ be obtained from~$T$ by contracting each component $C$ of $T-E_{tt'}$ to a single vertex~$t_C$.
	We set $V_{t_C}:=\bigcup_{s\in C}V_s$ and set $\Vcal'$ be the set of those sets~$V_{t_C}$.
	It is easy to see that $(T',\Vcal')$ is a basic \td\ with respect to~$\Gamma$: the only non-trivial requirement is that $(T',\Vcal')$ distinguishes at least two ends.
	But this follows from the fact that $V_t\cap V_{t'}$ separates two ends.
\end{proof}

Let us combine our results on connected and basic \td s.

\begin{coro}\label{compatible}
	Let~$\G$ be a group acting on a connected locally finite graph~$G$ with at least two ends.
	Then the following hold.
	\begin{enumerate}[label=(\roman*)]
	\item There is a connected basic \td\ of~$G$ with respect to~$\Gamma$.
	\item If $(T,(V_t)_{t\in V(T)})$ is a basic \td\ of~$G$ with respect to~$\Gamma$, then there is a connected basic \td\ $(T,(V'_t)_{t\in V(T)})$ of~$G$ with respect to~$\G$ such that $V_t\sub V'_t$ for every $t\in V(T)$.
	\end{enumerate}
\end{coro}

\begin{proof}
By Corollary~\ref{nicetreedecomposition}, there is a basic \td\ of~$G$.
Having a basic \td\ $(T,(V_t)_{t\in V(T)})$, Proposition~\ref{connected TD} and Lemma~\ref{nicetree} imply the existence of a connected basic \td\ $(T,(V'_t)_{t\in V(T)})$ with $V_t\sub V'_t$ for every $t\in V(T)$.
\end{proof}

Now we investigate some of the connections between the graphs and the parts of any of the connected basic \td s.
We start by showing that these \td s behave well with respect to the class of \qt\ graphs.

\begin{prop}\label{Aut(G)/V_t}\label{quasitransitivebags}
	Let~$\G$ be a group acting \qt ly on a connected locally finite graph $G$ with at least two ends and let $(T,\Vcal)$ be a connected basic \td\ of~$G$. 
	Then for each part $V_t\in \Vcal$ its stabilizer $\Gamma_{V_t}$ acts \qt ly on $G[V_t]$. 
\end{prop}

\begin{proof}
	If $u\in V_t$ does not lie in any adhesion set, then none of its images $v\in V_t$ under elements of~$\Gamma$ lie in an adhesion set.
	Hence, if $\gamma\in\G$ maps $u$ to~$v$, it must fix $V_t$ setwise, as it acts on $(T,\Vcal)$, so it lies in the stabilizer of~$V_t$.
	Thus, the intersection of $V_t$ with the $\Gamma$-orbit of~$u$ is the $\Gamma_{V_t}$-orbit of~$u$.
	
	Now we consider the vertices in an adhesion set $V_t\cap V_{t'}$.
	Let $V_t\cap V_s$ be another adhesion set.
	As $(T,\Vcal)$ is basic, there exists $\gamma\in\Gamma$ that maps $V_t\cap V_{t'}$ to $V_t\cap V_s$.
	If $\gamma$ stabilizes $V_t$, all vertices of $V_t\cap V_s$ lie in $\Gamma_{V_t}$-orbits of the vertices of $V_t\cap V_{t'}$.
	Let us assume that $\gamma$ does not stabilize~$V_t$ and let $V_t\cap V_{s'}$ be another adhesion set such that the element $\gamma'\in\Gamma$ that maps $V_t\cap V_{t'}$ to $V_t\cap V_{s'}$ does not stabilize $V_t$.
	Then $\gamma'\gamma\inv$ maps $V_t\cap V_s$ to $V_t\cap V_{s'}$ and stabilizes $V_t$.
	We conclude that the number of $\Gamma_{V_t}$-orbits of vertices in adhesion sets of~$V_t$ is at most twice the number of $\Gamma$-orbits of vertices in adhesion sets of~$V_t$.
\end{proof}

Subtrees of connected basic \td s that contain a common adhesion set cannot be to large as the following lemma shows.

\begin{lemma}\label{lem_TXfiniteDiam}
Let $\Gamma$ be a group acting \qt ly on a connected  locally finite graph~$G$ with at least two ends and let $(T,\Vcal)$ be a connected basic \td\ of~$G$ with respect to~$\Gamma$.
For an adhesion set~$X$ let $T_X$ be the maximal subtree of~$T$ such that $X\sub V_t$ for all $t\in V(T_X)$.
Then the diameter of~$T_X$ is at most~$2$.
\end{lemma}

\begin{proof}
Suppose the diameter of~$T_X$ is at least~$3$.
We have $V_t\cap V_{t'}=X$ for every $tt'\in E(T_X)$ since $X$ is contained in every adhesion set $V_t\cap V_{t'}$ and since they all have the same size.
Let $R=\ldots t_0t_1\ldots$ be a maximal path in~$T_X$.
We shall show that~$R$ is a double ray.

Let us suppose that $t_{i+3}$ is the last vertex on~$R$.
As $(T,\Vcal)$ is basic, we find $\gamma\in\Gamma$ such that $\gamma(t_it_{i+1})=t_{i+2}t_{i+3}$.
Note that $\gamma$ fixes $X=V_{t_{i}}\cap V_{t_{i+1}}=V_{t_{i+2}}\cap V_{t_{i+3}}$ setwise.
If $\gamma(t_i)=t_{i+2}$, then $\gamma(t_{i+2})$ is a neighbour of~$t_{i+3}$ distinct from~$t_{i+2}$ that contains~$X$, a contradiction to the choice of~$i$.
If $\gamma(t_i)=t_{i+3}$, then $\gamma$ fixes the edge $t_{i+1}t_{i+2}$ but neither of its incident vertices.
Let $\gamma'\in\Gamma$ map $t_{i+1}t_{i+2}$ to $t_{i+2}t_{i+3}$.
Note that $\gamma'$ fixes~$X$ setwise, too.
Then either $\gamma'$ or $\gamma'\gamma$ maps $t_i$ to a neighbour of~$t_{i+3}$ distinct from~$t_{i+2}$.
This is again a contradiction, which shows that $R$ has no last vertex.
Analogously, $R$ has no first vertex.
So it is a double ray.

Note that the part of some node of~$T_X$ contains~$X$ properly as $G=X$ is finite otherwise.
But as $\Gamma$ acts transitively on $E(T)$, we have at most two $\Gamma$-orbits on $V(T)$.
Hence infinitely many parts of~$R$ contain $X$ properly.
Thus and since each $V_{t_i}$ is connected, one vertex of~$X$ must have infinitely many neighbours.
This contradiction to local finiteness shows the assertion.
\end{proof}

Our next result is a characterisation of the finite parts of a connected basic \td.

\begin{prop}\label{BagDegree}
	Let~$\G$ be a group acting \qt ly on a connected  locally finite graph~$G$ with at least two ends and let $(T,\Vcal)$ be a connected basic \td\ of~$G$.
	Then the degree of a node $t \in V(T)$ is finite if and only if $V_t$ is finite. 
\end{prop}

\begin{proof}
	Note that each vertex lies in only finitely many adhesion sets as we only have one orbit of adhesion sets and as $G$ is locally finite.
	So if~$V_t$ is finite, then the degree of~$t$ is finite, too.
	
	Now let us assume that the degree of~$t$ is finite.
	Let $U$ be a subset of~$V_t$ that consists of one vertex from each $\Gamma_{V_t}$-orbit that meets~$V_t$.
	By Proposition~\ref{quasitransitivebags} the set~$U$ is finite.
	The vertices in~$U$ have bounded distance to the union $W$ of all adhesion sets in~$V_t$.
	As they meet all $\Gamma_{V_t}$-orbits and $\Gamma_{V_t}$ fixes~$W$ setwise, all vertices in~$V_t$ have bounded distance to~$W$.
	Note that $W$ is finite as $t$ has finite degree.
	Since $G$ is locally finite, $V_t$ must be finite.
\end{proof}

Let $(T,\Vcal)$ be a \td\ of a graph~$G$.
We say that an end $\eta$ of~$T$ \emph{captures} an end $\omega$ of~$G$ if for every ray $R=t_1,t_2,\ldots$ in~$\eta$ the union $\bigcup_{i\in\N}V_{t_i}$ captures~$\omega$ and a node of~$T$ \emph{captures} $\omega$ if its part does so.

Let us now investigate where the ends of~$G$ lie in $(T,\Vcal)$.

\begin{prop}
	\label{capture thick ends_new}\label{capture thick ends}
	Let~$G$ be a graph and let~$(T,\Vcal)$ be a connected \td\ of~$G$ such that the maximum size of its adhesion sets is at most $k\in\N$. 
	Then the following holds.
	\begin{enumerate}[label=(\roman*)]
	\item\label{itm_cte_1} Each end of~$G$ is captured either by an end or by a node of~$T$.
	\item\label{itm_cte_2} Every thick end of~$G$ is captured by a node of~$T$.
	\item\label{itm_cte_3} Every end of~$T$ captures a unique thin end of~$G$, which has degree at most~$k$.
	\item\label{itm_cte_4} Assume that $\Gamma$ acts \qt ly on $G$ and that $(T,\Vcal)$ is $\Gamma$-invariant with only finitely many $\Gamma$-orbits on~$E(T)$. Every end of~$G$ that is captured by a node $t\in V(T)$ corresponds to a unique end of~$G[V_t]$.\footnote{This shall mean that for every end $\omega$ of~$G$ that is captured by $t\in V(T)$ there is a unique end $\omega_t$ of $G[V_t]$ with $\omega_t\sub\omega$.}
	\end{enumerate}
\end{prop}

\begin{proof}
	Let~$\omega$ be an end of~$G$ and let $Q,R$ be two rays in~$\omega$.
	For an edge $st\in E(T)$ let $T_s$ and $T_t$ be the subtrees of $T-st$ with $s\in V(T_s)$ and $t\in V(T_t)$.	
	If the ray $Q$ has all but finitely many vertices in $\bigcup _{x\in V(T_s)}V_x$ and $R$ has all but finitely many vertices in $\bigcup_{x\in V(T_t)}V_x$ or vice versa, then we have a contradiction as $Q$ and $R$ cannot lie in the same end if they have tails that are separated by the finite vertex set $V_s\cap V_t$.
	We now orient the edge $st$ from $s$ to~$t$ if $Q$ and~$R$ lie in $\bigcup_{x\in V(T_t)}V_x$ eventually and we orient it from $t$ to~$s$ if the rays lie in $\bigcup _{x\in V(T_s)}V_x$ eventually.
	Obviously, every node of~$T$ has at most one outgoing edge.
	Let $t_Q, t_R$ be nodes of~$T$ such that the first vertex of~$Q$ lies in~$V_{t_Q}$, and the first vertex of~$R$ lies in~$V_{t_R}$, and let $P_Q$ and $P_R$ be the maximal (perhaps infinite) directed paths in our orientation of~$T$ that start at~$t_Q$ and $t_R$, respectively.
	Note that if $P_Q$ and $P_R$ meet at a vertex, they continue in the same way.
	Thus, if they meet, they either end at a common vertex or have a common infinite subpath.
	We shall show that $P_Q$ and $P_R$ meet.
	Let $P$ be the $t_Q$-$t_R$ path in~$T$.
	Then there is a unique sink $x$ on it as every node of~$T$ has at most one outgoing edge.
	This sink is a common node of~$P_Q$ and~$P_R$.
	If $P_Q$ and $P_R$ end at a node, this node captures~$\omega$ and if they share a common infinite subpath, this is a ray whose end captures~$\omega$.
	We proved \ref{itm_cte_1}.
	
	Now let us assume that $\omega$ has degree at least~$k+1$.
	Then there are $k+1$ pairwise disjoint rays $R_1,\ldots, R_{k+1}$ in~$\omega$.
	Let $t_i, P_i$ be a node and a path of~$T$ defined for $R_i$ as we defined $t_R$ and $P_R$ for the ray~$R$.
	By an easy induction, we can extend the above argument that $P_Q$ and $P_R$ meet to obtain that all $P_1,\ldots, P_{k+1}$ have a common node~$x$.
	Let us suppose that $\omega$ is captured by an end $\eta$ of~$T$.
	Let $y$ be the node of~$T$ that is adjacent to~$x$ and that separates $x$ and~$\eta$.
	Then all rays $R_i$ must contain a vertex of $V_x\cap V_y$.
	This is not possible as $V_x\cap V_y$ contains at most~$k$ vertices and the rays $R_i$ are disjoint.
	This contradiction shows \ref{itm_cte_2} and the second part of~\ref{itm_cte_3}.

	Let $R,Q$ be two rays that lie in ends of~$G$ that are captured by the same end $\eta$ of~$T$.
	With the notations $P_Q,P_R$ as above, the intersection $P_Q\cap P_R$ is a ray in~$\omega$.
	As $G$ is locally finite and $(T,\Vcal)$ is a connected \td, there are infinitely many disjoint paths between $Q$ and~$R$ and thus, they are equivalent and lie in the same end of~$G$.
	This proves~\ref{itm_cte_3}.
	
	To prove~\ref{itm_cte_4}, let us assume that $\Gamma$ acts \qt ly on $G$ and has finitely many orbits on the edges of the decomposition tree $T$.
	Let $\omega$ be an end of~$G$ that is captured by a node $t\in V(T)$ and let $R$ be a ray in~$\omega$ that starts at a vertex in~$V_t$.
	Since $V_t$ captures~$\omega$, there are infinitely many vertices of~$V_t$ on~$R$.
	Whenever $R$ leaves~$V_t$ through an adhesion set, it must reenter it through the same adhesion set by Remark~\ref{rem_adhSeps}.
	We replace every such subpath~$P$, where the end vertices of~$P$ lie in a common adhesion set and the inner vertices of~$P$ lie outside of~$V_t$, by a geodesic in~$G[V_t]$ between the end vertices of~$P$.
	We end up with a walk $W$ with the same starting vertex as~$R$.
	We shall see that $W$ contains a one-way infinite path.
	First, we recursively delete closed subwalks of~$W$ to end up with a path~$R'$.
	Since $G$ is locally finite and $R$ meets $V_t$ infinitely often, $R$ contains vertices of~$V_t$ that are arbitrarily far away from the starting vertex of~$R$.
	As we only took geodesics to replace the subpaths of~$R$ that were outside of~$V_t$ and as $\Gamma$ acts on $(T,\Vcal)$ with only finitely many orbits on the edges of~$T$, these replacement paths have a bounded length.
	Hence, $W$ eventually leaves every ball of finite diameter around its starting vertex.
	This implies that $R'$ is a ray.
	Obviously, $R$ and $R'$ are equivalent.
	Thus, $G[V_t]$ contains a ray in~$\omega$.
	Let $\omega_t$ be the end of $G[V_t]$ that contains~$R'$ and let $Q$ be a ray in~$\omega_t$.
	Since no finite separator can separate $Q$ and $R'$ in~$G[V_t]$, the rays are also equivalent in~$G$.
	Thus, we have shown $\omega_t\sub\omega$.
	
	Let $\omega_t'$ be an end in $G[V_t]$ different from~$\omega_t$, let $S$ be a finite subset of~$V_t$ that separates $\omega_t$ from $\omega_t'$, and let $P$ be a path in $G$ connecting vertices in different components of $G[V_t]-S$.
	As before, whenever $P$ leaves $V_t$ through an adhesion set, it must reenter it through the same adhesion set by Remark~\ref{rem_adhSeps}.
	We again replace every such subpath, where the end vertices lie in a common adhesion set and the inner vertices lie outside of~$V_t$, by a geodesic in~$G[V_t]$ to obtain a walk $P'$ in $G[V_t]$.
	Since $P$ and $P'$ have the same endpoints and $P'$ must meet $S$, we know that $P$ either contains a vertex in $S$, or it contains a vertex in an adhesion set which meets $S$. 
	Let $S'$ be the set containing all vertices of~$S$ and all vertices contained in adhesion sets that meet~$S$. 
	There are only finitely many orbits of vertices in adhesion sets, hence there is an upper bound on the diameter of the adhesion sets. 
	Since $S$ is finite and $G$ is locally finite, this implies that $S'$ is finite. 
	By definition, there is no path in $G-S'$ connecting vertices in different components of $G[V_t]-S$. 
	In particular, $S'$ separates every ray in $\omega_t$ from every ray in $\omega_t'$, and hence \ref{itm_cte_4} holds.
\end{proof}

\section{Tree amalgamations}
\label{Sec_Splitting}

In this section, we prove our main result, Theorem~\ref{stallingsgraphIntro}.
But before we move on to that proof, we first have to state some definitions, in particular, the main definition: \ta s, a notion introduced by Mohar~\cite{Mohar06}.
After we stated those definitions, we compare \ta s and connected basic \td s.

For the definition of \ta s, let~$G_1$ and~$G_2$ be graphs. 
Let $(S_k^i)_{k \in I_i}$ be a family of subsets of~$V(G_i)$.
Assume that all sets~$S_k^i$ have the same cardinality and that the index sets $I_1$ and $I_2$ are disjoint.
For all $k \in I_1$ and $\ell \in I_2$, let $\phi_{k \ell} \colon S_k^1 \rightarrow S_\ell^2$ be a bijection and let $\phi_{\ell k} = \phi_{k\ell}^{-1}$. 
We call the maps $\phi_{k\ell}$ and $\phi_{\ell k}$ \emph{bonding maps}.

Let~$T$ be a \emph{$(|I_1|,|I_2|)$-semiregular} tree, that is, a tree in which for the canonical bipartition~$\{V_1,V_2\}$ of~$V(T)$  the vertices in~$V_i$ all have degree~$|I_i|$. 
Denote by $D(T)$ the set obtained from the edge set of $T$ by replacing every edge $xy$ by two directed edges $\fwd {xy}$ and $\fwd {yx}$. 
For a directed edge $\fwd{e} = \fwd{xy} \in D(T)$, we denote by $\bwd e = \fwd{yx}$ the edge with the reversed orientation.
Let $f\colon D(T) \to I_1 \cup I_2$ be a labelling, such that for every $t \in V_i$, the labels of edges starting at $t$ are in bijection to $I_i$.

For every $i\in\{1,2\}$ and for every $t \in V_i$, take a copy~$G_t$ of the graph~$G_i$. 
Denote by~${S^t_k}$ the corresponding copies of~$S_k^i$ in~$V(G_t)$. 
Let us take the disjoint union of the graphs $G_t$ for all $t \in V(T)$.
For every edge $\fwd e = \fwd{st}$ with $f(\fwd e) = k$ and $f(\bwd e) = \ell$ we identify each vertex $x$ in the copy of $S^{s}_k$ with the vertex $\phi_{k\ell}(x)$ in $S^{t}_\ell$.
Note that this does not depend on the orientation we pick for $\fwd e$, since $\phi_{\ell k} = \phi_{k\ell}^{-1}$.
The resulting graph is called the \emph{tree amalgamation} of the graphs~$G_1$ and~$G_2$ over the \emph{connecting tree}~$T$ and is denoted by $G_1\ast G_2$ or by $G_1\free_T G_2$ if we want to specify the tree.

In the context of tree amalgamations the sets $S_k^i$ are called the \emph{adhesion sets} of the tree amalgamation. 
More specifically, the sets $S_k^1$ are the adhesion sets of~$G_1$ and the sets $S_k^2$ are the adhesion sets of~$G_2$. 
If the adhesion sets of a tree amalgamation are finite, then this tree amalgamation has \emph{finite adhesion}. 
We call a \ta\ $G_1\ast_T G_2$ \emph{trivial} if for some $t\in V(T)$ the canonical map that maps the vertices $x\in V(G_t)$ to the vertices of $G_1\ast_T G_2$ that is obtained from~$x$ by all the identifications is a bijection.
Note that if the \ta\ has finite adhesion, it is trivial if $V(G_i)$ is the only adhesion set of~$G_i$ and $|I_i|=1$ for some $i \in \{1,2\}$.

We remark that the map described in the definition of a trivial \ta\ does not induce a graph isomorphism $G_t\to G_1\ast_T G_2$: it is a bijection $V(G_t)\to V(G_1\ast_T G_2)$ but need not induce a bijection $E(G_t)\to E(G_1\ast_T G_2)$.

The \emph{identification length} of a vertex $x\in V(G_1\free_T G_2)$ is the diameter of the subtree $T'$ of~$T$ induced by all nodes $t$ for which a vertex of $G_t$ is identified with $x$.
The \emph{identification length} of the \ta\ is the supremum of the identification lengths of its vertices.
The \ta\ has \emph{finite identification length} if the identification length is finite.

We remark that in Mohar's definition of a \ta\ \cite{Mohar06} the identification length is always at most~$2$.
But apart from this, our definition is equivalent to his.

It is worth noting that every tree amalgamation gives rise to a tree decomposition in the following sense.
\begin{remark}\label{TA_implies_TD}
Let $G$ be a graph.
If $G$ is a \ta\ $G_1\ast_T G_2$ of finite adhesion, then there is a naturally defined \td\ of~$G$: for $t\in V(T)$ let $V_t$ be the set obtained from $V(G_t)$ after all identifications in $G_1\ast G_2$.
Set $\Vcal:=\{V_t\mid t\in V(T)\}$.
Obviously, all vertices of~$G$ lie in $\bigcup_{t\in V(T)}V_t$ and for each edge there is some $V_t\in\Vcal$ containing~it.
Property (T3) of a \td\ is satisfied as the copies $G_i^v$ are arranged in a treelike way and as identifications to obtain a vertex take place in subtrees of~$T$.
So $(T,\Vcal)$ is a \td.
If $G_1\ast_T G_2$ has finite adhesion, so does $(T,\Vcal)$.
If the \ta\ is non-trivial, then $T$ has at least two ends and so does $G$.
Also, $(T,\Vcal)$ distinguishes two ends of~$G$: those that are captured by ends of~$T$.
\end{remark}

So far, the \ta s do not interact with any group actions on $G_1$ and~$G_2$.
In particular, it is easy to construct a \ta\ of two \qt\ graphs that is not \qt: e.\,g.\ take as~$G_1$ a double ray and as~$G_2$ a finite non-trivial graph. Let $G_1$ have precisely two adhesion sets and $G_2$ at least two, all of size~$1$. The \ta\ $G_1\ast G_2$ is not \qt.

In the following, we describe some conditions on \ta s which will ensure that \ta s of \qt\ graphs are again \qt, see Lemma~\ref{lem_TAareQT}.

Let $\Gamma_i$ be a group acting on~$G_i$ for $i=1,2$, let $t \in V_i$, let $\gamma \in \Gamma_i$ and let $j\in\{1,2\}\sm\{i\}$.
We say that the \ta\ \emph{respects $\gamma$}, if there is a permutation $\pi$ of $I_i$ such that for every $k \in I_i$ there is $\ell \in I_j$ such that 
\[
\phi_{k\ell} = \phi_{\pi(k)\ell} \circ \gamma\mid_{S_k}.
\]
Note that this in particular implies that $\gamma(S_k) = S_{\pi(k)}$.
The \ta\ \emph{respects $\Gamma_i$} if it respects every $\gamma \in \Gamma_i$.

Let $k\in I_i$ and let $\ell,\ell'\in I_j$. 
We call the bonding maps from $k$ to $\ell$ and $\ell'$ \emph{consistent} if there is $\gamma \in \Gamma_j$ such that 
\[
    \phi_{k\ell} = \gamma \circ \phi_{k\ell'}.
\]
We say that the bonding maps between two sets $J_1 \subseteq I_1$ and $J_2\subseteq I_2$ are \emph{consistent}, if they are consistent for any $i \in \{1,2\}$, $k \in J_i$, and $\ell,\ell' \in J_j$.

We say that the \ta\ $G_1\ast G_2$ is of \emph{Type 1 respecting the actions of~$\Gamma_1$ and $\Gamma_2$} or \emph{$(G_1,\Gamma_1)\ast(G_2,\Gamma_2)$ is a \ta\ of Type 1} for short if the following holds:
\begin{enumerate}[label=(\roman*)]
\item \label{itm:T1respaction} The \ta\ respects $\Gamma_1$ and $\Gamma_2$.
\item \label{itm:T1esim} The bonding maps between $I_1$ and $I_2$ are consistent.
\end{enumerate}

We say that the \ta\ $G_1\ast G_2$ is of \emph{Type 2 respecting the actions of~$\Gamma_1$ and $\Gamma_2$} or \emph{$(G_1,\Gamma_1)\ast(G_2,\Gamma_2)$ is a \ta\ of Type 2} for short if the following holds:
\begin{enumerate}[label=(\roman*)]
\item[(o)] \label{itm:T2prelim} $G_1=G_2=:G$, $\Gamma_1=\Gamma_2=:\Gamma$, and  $I_1=I_2=:I$,\footnote{Technically this is not allowed, in particular since for the definition of $\phi_{k\ell}$ we needed $I_1$ and $I_2$ to be disjoint. These technicalities can be easily dealt with by an appropriate notion of isomorphism the details of which we leave to the reader.}
and there is $J \subseteq I$ such that $f(\fwd e) \in J$, if and only if $f(\bwd e) \notin J$.
\item \label{itm:T2respaction} The \ta\ respects $\Gamma$.
\item \label{itm:T2esim} The bonding maps between $J$ and $I\setminus J$ are consistent.
\end{enumerate}
In this second case also say that $G_1\ast G_2 = G \ast G$ is a \emph{\ta\ of~$G$ with itself}.

We say that $G_1\ast G_2$ is a \ta\ \emph{respecting the actions of $\Gamma_1$ and $\Gamma_2$} if it is of either Type~1 or Type~2 respecting the actions $\Gamma_1$ and $\Gamma_2$ and we speak about the \emph{\ta} $(G_1,\Gamma_1)\ast(G_2,\Gamma_2)$.

Note that conditions \ref{itm:T1respaction} and \ref{itm:T1esim} in both cases do not depend on the specific labelling of the tree.
This is no coincidence. 
In fact we will show that any two legal labellings of $D(T)$ give isomorphic \ta s, see Lemma~\ref{lem_TAareQT}.
Furthermore, any $\gamma \in \Gamma_i$ (interpreted as an isomorphism between parts of two such \ta s) can be extended to an isomorphism of the \ta s, which also implies that the \ta s obtained this way are always \qt .

Before we turn to the proof of these facts, we need some notation.
A \emph{legally labelled star centred at $V_i$} is a function $\ell$ from $I_i$ to $I_j$.
If the \ta\ is of Type 2, we further require that $\ell (k) \in J$ if and only if $k \notin J$.
Informally, think of this as a star whose labels on directed edges could appear on a subtree of $T$ induced by a vertex $t \in V_i$ and its neighbours:
for $\fwd e$ with label $k$, the value $\ell(k)$ tells us the label of $\bwd e$.

An isomorphism of two legally labelled stars $\ell, \ell'$ is a triple $(\gamma,\pi,(\gamma_k)_{k\in I_i})$ consisting of some $\gamma \in \Gamma_i$, a permutation $\pi$ of $I_i$, and a family $(\gamma_k)_{k \in I_i}$ of elements of $\Gamma_j$ such that for every $k \in I_i$
\[
    \phi_{k,\ell(k)} = \gamma_k \circ \phi_{\pi(k)\ell'(\pi(k))}\circ \gamma \mid_{S_k}.
\]
In our interpretation of legally labelled stars as subtrees of $T$, this corresponds to an isomorphism of the corresponding subgraphs of the \ta .

\begin{prop}
    \label{prop:isomorphicstars}
    Let $\ell, \ell'$ be two legally labelled stars with respect to a \ta\ $(G_1,\Gamma_1)\ast_T(G_2,\Gamma_2)$ centred at $V_i$ and let $\gamma \in \Gamma_i$. Then $\gamma$ extends to an isomorphism $(\gamma',\pi,(\gamma_k)_{k\in I_i})$ of $\ell$ and $\ell'$.
    Furthermore, if we are given $\tilde k, \tilde k' \in I_i$ and $\tilde \gamma_k \in \Gamma_j$ such that
    \[
    \phi_{\tilde k,\ell(\tilde k)} = \tilde \gamma_k \circ \phi_{\tilde k'\ell'(\tilde k')}\circ \gamma \mid_{S_k},
    \]
    then we can choose $\pi(\tilde k) = \tilde k'$ and $ \gamma_{\tilde k} = \tilde \gamma_k$.
\end{prop}

\begin{proof}
Since the \ta\ respects $\gamma$, there are $\pi$ and $\bar\ell\colon I_i \to I_j$ such that 
\[
    \phi_{k \bar \ell(k)} = \phi_{\pi(k) \bar \ell(k)} \circ \gamma\mid_{S_k}.
\]
Let $\gamma_k\in\Gamma_j$ be such that $\phi_{k\ell(k)} = \gamma_k \circ \phi_{k \bar \ell(k)}$, and let $\gamma'_k\in\Gamma_j$ be such that $\phi_{\pi(k) \bar\ell(k)} = \gamma'_k \circ \phi_{\pi(k) \ell'(\pi(k))}$. These exist by \ref{itm:T1esim}; for Type 2 recall that by the definition of legally labelled stars $k \in J$ if and only if $\ell(k) \notin J$. Now clearly
\[
    \phi_{k\ell(k)} = \gamma_k \circ \gamma'_k \circ \phi_{\pi(k) \ell'(\pi(k))} \circ \gamma \mid_{S_k},
\]
thus showing that the two stars are isomorphic.

For the second part, let $(\gamma', \pi,(\gamma_k)_{k \in I_i})$ be an isomorphism between $\ell$ and $\ell'$. Let $\tilde k'' = \pi^{-1}(\tilde k')$. Define $\tau(\tilde k) = \tilde k'$ and $\tau(\tilde k'') = \pi(\tilde k)$. Let $ \delta_{\tilde k} = \tilde \gamma_k$ and let
\[
    \delta_{\tilde k''} = \gamma_{\tilde k''} \circ \tilde \gamma_k^{-1} \circ \gamma_{\tilde k}.
\]
For the remaining $k \in I_i$, let $\tau(k) = \pi(k)$ and $\delta_k = \gamma_k$. It is straightforward to check that $\gamma$, $\tau$, and $(\delta_k)_{k \in I_i}$ define an isomorphism between $\ell$ and $\ell'$ with the desired properties.
\end{proof}

\begin{lemma}\label{lem_TAareQT}
Let $G_1$ and $G_2$ be connected locally finite graphs and let $\Gamma_i$ be a group acting \qt ly on~$G_i$ for $i=1,2$. Then the \ta\ $(G_1,\Gamma_1)\ast_T(G_2,\Gamma_2)$ is \qt\ and independent (up to isomorphism) of the particular labelling of $T$.
\end{lemma}

\begin{proof}
Let $T$ and $T'$ be two labelled trees giving rise to tree amalgamations $G=(G_1,\Gamma_1)\ast_T(G_2,\Gamma_2)$ and $G=(G_1,\Gamma_1)\ast_{T'}(G_2,\Gamma_2)$, respectively, such that the adhesion sets as well as the bonding maps for both \ta s are the same.
Let $t \in V(T)$ and let $t'\in V(T')$ such that $G_t$ and $G_{t'}$ are both isomorphic to $G_i$. Let $\gamma_t \in \Gamma_i$. 
We claim that there is an isomorphism $\bar \gamma\colon G\to G'$ such that 
\[
\bar\gamma \mid_{G_t} = \id_{t'} \circ \gamma_t \circ \id_t^{-1},
\]
where $\id_t$ and $\id_{t'}$ denote the canonical isomorphisms from $G_i$ to $G_t$ and $G_{t'}$ respectively. Clearly, the lemma follows from this claim. 

For the proof of the claim define the star around $s \in V(T)$ by the map $\ell_s$ mapping $k$ to the label of $\bwd {e_k}$, where $\fwd{e_k}$ is the unique edge with label $k$ starting at $s$.
By Proposition \ref{prop:isomorphicstars}, there are a bijection $\pi\colon N(t) \to N(t')$ and a family $(\gamma_s \in \Gamma_j)_{s \in N(t)}$ which extend $\gamma_t$ to an isomorphism of the stars around $t$ and $t'$.
Iteratively apply Proposition \ref{prop:isomorphicstars} to vertices at distance $n=\{1,2,3,\dots\}$ from~$t$. 
We obtain an isomorphism $\pi\colon T\to T'$ and maps $\gamma_s \in \Gamma_i$ for each $s \in V_i$ such that the restriction of $\pi$ to $s$ and its neighbours and the corresponding maps $\gamma_x$ form an isomorphism between the stars at $s$ and $\pi(s)$. 

For $v \in V(G_s)$, define $\bar \gamma(v) = \id_{\pi(s)} \circ \gamma_s \circ \id_s^{-1} (v)$. 
Note that for any edge $ss'$ the two concurring definitions given for vertices of~$G_s$ and~$G_{s'}$ that get identified for the \ta\ coincide.
Hence $\bar \gamma$ is well defined, and since it obviously maps edges to edges and non-edges to non-edges, it is the desired isomorphism.
\end{proof}

A closer inspection of the proof of Lemma \ref{lem_TAareQT} together with Remark \ref{TA_implies_TD} shows that \ta s respecting the actions of \qt\ groups give rise to basic \td s of $(G_1,\Gamma_1)\ast(G_2,\Gamma_2)$.
The following lemma shows that the converse also holds, that is, basic \td s of \qt\ graphs give rise to \ta s respecting the actions of some \qt\ group on the parts.

\begin{lemma}
	\label{TD_implies_TA}
	Let~$\G$ be a group acting \qt ly on a connected locally finite graph~$G$ and let ${(T,\Vcal)}$ be a connected basic \td\ of~$G$ with respect to~$\G$.
	Then one of the following holds.
	\begin{enumerate}[label=(\arabic*)]
	\item There are $V_t, V_{t'} \in \Vcal$ such that $G$ is a non-trivial \ta\ \[G[V_t] \free_T G[V_{t'}]\]
	of Type 1 respecting the actions of the stabilisers of $G[V_t]$ and $G[V_{t'}]$ in~$\Gamma$. 
	\item There is $V_t\in\Vcal$ such that $G$ is a non-trivial \ta\ \[G[V_t]\ast_TG[V_t]\]
	of Type 2 respecting the actions of the stabiliser of $G[V_t]$ in~$\Gamma$.
	\end{enumerate} 
\end{lemma}
\begin{proof}
    Choose an oriented edge $\fwd {e_0} \in D(T)$. 
    We say that $\fwd e \in D(T)$ is positively oriented, if there is $\gamma \in \Gamma$ mapping $\fwd{e_0}$ to $\fwd e$. 
    Otherwise we say that $\fwd e$ is negatively oriented. 
    If $\Gamma$ contains an element that reverses an edge of $T$, then let $\Gamma'$ be the subgroup preserving the bipartition of $T$.
    This subgroup has index $2$, and still acts \qt ly on $G$ and transitively on edges of $T$.
    Hence we can without loss of generality assume that no element of $\Gamma$ swaps the endpoints of an edge, and thus every edge is either positively or negatively oriented, but not both.
    
    Let $s$ and $t$ be the start and end point of $\fwd{e_0}$ respectively.
    Let $(\fwd{e_k})_{k \in K}$ be the positively oriented edges starting at $s$ and let $(\fwd{e_\ell})_{\ell \in L}$ be the negatively oriented edges starting at $t$.
    Without loss of generality, assume that $K$ and $L$ are disjoint, and that $\fwd {e_0} = \fwd{e_{k_0}} = \bwd{e_{\ell_0}}$.
    For every $k \in K$ pick $\gamma_k \in \Gamma$ which maps $\fwd {e_0}$ to $\fwd{e_k}$ (with $\gamma_{k_0} = \id$).
    For every $\ell \in L$ pick $\gamma_\ell \in \Gamma$ which maps $\bwd{e_0}$ to $\fwd {e_\ell}$ (with $\gamma_{\ell_0} = \id$).
    If there is an element of $\Gamma$ that maps $s$ to $t$, then fix such an automorphism $\gamma_{st}$, and for $k \in K, \ell \in L$ let $\gamma'_k = \gamma_k \circ \gamma_{st}$ and $\gamma'_\ell = \gamma_{\ell} \circ \gamma_{st}^{-1}$.
    
    Note that $e_0$ can be mapped to any edge incident to $e_0$ by a unique element of the form $\gamma_k$ or $\gamma'_k$ for some $k \in K \cup L$.
    For an arbitrary edge $e$, let $e'$ be the first edge of the path connecting $e$ to $e_0$.
    If $\gamma_{e'} \in \Gamma$ maps $e_0$ to $e'$, then by the above remark there is a unique element $\delta_{e}$ of the form $\gamma_k$ or $\gamma'_k$ such that $\gamma_e \circ \delta_e$ maps $e_0$ to $e$.
    Use this to inductively construct (starting from $\delta_{e_0} = \id$) for each $e \in E(T)$ an automorphism $\gamma_e \in \Gamma$ such that $\gamma_e(e_0) = e$.
    Let $\fwd e$ be the orientation of $e$ pointing away from $e_0$ if $e\neq e_0$ and $\fwd e=\fwd{e_0}$ otherwise.
    Define the label $f(\fwd e$) to be the unique $k \in K \cup L$ such that the $\delta_e$ from above is $\gamma_k$ or $\gamma_k'$.
    Note that $k \in K$ if and only if $\fwd e$ is positively oriented.
    In this case define $f(\bwd e) = \ell_0$, otherwise define $f(\bwd e) = k_0$.
    
    The following observation will be useful later. 
    Let $v$ be a vertex of $T$, and let $\fwd e$ be the first edge of the path from $v$ to $e_0$ (in case $v$ is $s$ or $t$ this is an orientation of~$e_0$).
    Let $\Delta_v =\{\delta_f\mid v \in f, f \neq e\}$.
    \begin{itemize}
        \item If all edges starting at $v$ are positively (resp.\ negatively) oriented, then $\Delta_v = \{\gamma_k \mid k_0 \neq k \in K\}$ (resp.\ $\Delta_v = \{\gamma_\ell \mid \ell_0 \neq \ell \in L\})$.
        \item Otherwise, if $\fwd e$ is positively (resp.\ negatively) oriented, then $\Delta_v = \{\gamma_k, \gamma'_\ell \mid k_0 \neq k \in K, \ell \in \ell\}$ (resp.\ $\Delta_v = \{\gamma_k',\gamma_\ell \mid k \in K, \ell_0 \neq \ell \in L\})$.
    \end{itemize}
    In particular, taking into account the label of $\fwd e$, in the first case the edges starting at $v$ are labelled bijectively by $K$ (resp.\ L), while in the second case they are labelled bijectively by $K \cup L$.

    Next we show how this labelling defines a \ta . 
    First assume there is no automorphism $\gamma \in \Gamma$ mapping $s$ to $t$.
    Then all positively oriented edges must point from $V_1$ to $V_2$, where $V_1 \cup V_2$ is the bipartition of $T$ with $s \in V_1$---this corresponds to the first case in the above observation.
    Let $G_1$ be isomorphic to $G[V_s]$, and let $G_2$ be isomorphic to $G[V_t]$.
    Let $\id_s$ and $\id_t$ be the respective isomorphisms.
    
    For the definition of the adhesion sets let $I_1 = K$ and $I_2 = L$.
    For $k \in K$ let $t_k$ be the endpoint of $\fwd{e_k}$ and define $S_k = \id_s^{-1}(V_s \cap V_{t_k})$.
    Similarly, for $\ell \in L$, let $s_\ell$ be the endpoint of $\fwd{e_\ell}$ and define $S_k = \id_t^{-1}(V_t \cap V_{s_\ell})$. Finally, define the adhesion maps by
    $\phi_{k\ell} = \id_t^{-1} \circ \gamma_\ell \circ \gamma_k^{-1} \circ \id_s \mid_{S_k}$.
    
    The labels of directed edges starting at each vertex are in bijection to $K$ or $L$ depending on whether the vertex is in $V_1$ or $V_2$.
    Hence the above information together with the labelling defines a \ta .
    If $\Gamma_i$ is a group acting on $G_i$ in the same way as the setwise stabiliser (in $\Gamma$) of~$G_s$, of~$G_t$ acts on $G_s$, on $G_t$ respectively, then it is straightforward to verify that this \ta\ is of Type 1 respecting the actions. 
    Note that the possible replacement of~$\Gamma$ by~$\Gamma'$ changes neither $\Gamma_1$ nor~$\Gamma_2$.
    
    It only remains to show that the \ta\ is isomorphic to~$G$.
    Let $e_v$ be the first edge on the path from $v\in V(T)$ to~$e_0$.
    If $v\in V_1$, then set $\id_v = \gamma_{e_v} \circ \id_s$. 
    Otherwise set $\id_v = \gamma_{e_v} \circ \id_t$.
    It is easy to verify that for an edge $e = uv$ with labels $f(\fwd e) = k, f(\bwd e) = \ell$ we have that $\id_v^{-1} \circ \id_u = \phi_{k\ell}$, and this clearly shows that the \ta\ is isomorphic to $G$.
    
    The proof in the case where there is $\gamma_{st}$ mapping $s$ to $t$ is very similar to the first case. 
    Define $G_1$ and $G_2$ as before, but make sure that $\gamma_{st} \circ \id_s = \id_t$.
    This ensures that the actions $\Gamma_1$ on $G_1$ and $\Gamma_2$ on $G_2$ are the same, hence we can without loss of generality assume that $G_1 = G_2$ and $\Gamma_1 = \Gamma_2$.
    
    Set $I_1 = I_2 = K \cup L$ and let $J=K$. 
    Recall that $f(\fwd e) \in K$ if and only if $f(\bwd e) \in L$.
    Since we can map $s$ to $t$, there are positively and negatively oriented edges starting at each vertex, hence the labels of edges starting at any vertex are in bijection with $K \cup L$.
    Hence (o) for \ta s of Type 2 holds.
    Define the adhesion sets and adhesion maps exactly as above (but note that all adhesion sets end up in the same graph since $G_1= G_2$).
    This gives a \ta\ of Type 2 by construction which is isomorphic to $G$ by the same argument as above.
\end{proof}

Now we are ready to prove the main result of this section, the graph-theoretical analogue of Stallings' theorem, Theorem~\ref{stallingsgraphIntro}.
We are proving a slightly stronger version than the one we stated in the introduction.

\begin{thm}\label{stallingsgraph}
Let $\Gamma$ be a group acting \qt ly on a locally finite graph $G$ with more than one end. Then there are subgraphs $G_1,G_2$ of~$G$ and groups $\Gamma_1,\Gamma_2$ acting \qt ly on $G_1,G_2$, respectively, such that $G$ is a non-trivial \ta\ $(G_1,\Gamma_1)\ast(G_2,\Gamma_2)$ of finite adhesion and finite identification length.

Furthermore, $\Gamma_i$ can be chosen to be the setwise stabiliser of~$G_i$ in~$\Gamma$.
%
\end{thm}

\begin{proof}
By Corollary~\ref{compatible}, $G$ has a connected basic \td\ $(T,\Vcal)$. 
	Using Lemma~\ref{TD_implies_TA}, $G$ is a non-trivial \ta\ $(G_1,\Gamma_1)\ast_T (G_2,\Gamma_2)$, where $\Gamma_i$ is the setwise stabiliser of~$G_i$ in~$\Gamma$ for $i=1,2$.
	Proposition~\ref{quasitransitivebags} implies that $\Gamma_i$ acts \qt ly on~$G_i$ for $i=1,2$.
	It remains to show that the identification length is finite.
	Note that every vertex lies in only finitely many adhesion sets as all those adhesion sets lie in a common $\Gamma$-orbit and as $G$ is locally finite.
	This directly implies that the identification length is finite.
\end{proof}


\section{Accessible graphs}
\label{Sec_Accessible}

Let $G$ be a connected \qt\ locally finite graph with more than one end and let $\Gamma$ act \qt ly on~$G$.
We say that $G$ \emph{splits (non-trivially)} into connected \qt\ locally finite graphs $G_1,G_2$ if it is a non-trivial \ta\ $G = G_1 \free G_2$ of finite adhesion and if the \td\ defined by $G_1\free G_2$ (as in Remark~\ref{TA_implies_TD}) is basic with respect to~$\Gamma$.
Note that the stabilizer $\Gamma_i$ in~$\Gamma$ of~$G_i$ acts \qt ly on~$G_i$ by Proposition~\ref{quasitransitivebags}.
Now if one of the \emph{factors} $G_1$ or $G_2$ also has more than one end, we can split it with respect to~$\Gamma_i$, too.
We can continue this for every factor and call this a \emph{process of splittings}.
Note that it is important in a process of splittings to use the group action of the stabiliser of the factor in order to split the factor.
If we eventually end up with factors that are either finite or have at most one end, i.\,e.\ if the process of splittings terminates, we call the (multi-)set of these factors a \emph{terminal factorisation} of~$G$.
(Also, if $G$ is one-ended, we say it is a \emph{terminal factorisation} of itself.)
We call $G$ \emph{accessible} if it has a terminal factorisation.

\begin{remark}\label{rem_access2}
Let $G$ be an accessible connected \qt\ locally finite graph.
Then there are connected \qt\ locally finite graphs $G_1,\ldots$, $G_n$, $H_1,\ldots,H_{n-1}$ with $G=H_{n-1}$ and trees $T_1,\ldots, T_{n-1}$ such that the following hold:
\begin{enumerate}[label=(\roman*)]
\item every $G_i$ has at most one end;
\item for every $i\leq n-1$, the graph $H_i$ is a \ta\ $H\ast_{T_i}H'$ with respect to group actions of finite adhesion, where
\[
H,H'\in\{G_j\mid 1\leq j\leq n\}\cup\{H_j\mid 1\leq j<i\}.
\]
\end{enumerate}
\end{remark}

\begin{remark}\label{rem_access3}
Let $\Gcal_0$ be the class of all connected \qt\ locally finite graphs with at most one end.
For $i>0$, let $\Gcal_i$ be the class obtained by \ta s of finite adhesion of elements in $\bigcup_{j<i} \Gcal_j$.
Set $\Gcal:=\bigcup_{i\in\N}\Gcal_i$.
Then $\Gcal$ is the class of all accessible connected \qt\ locally finite graphs.
\end{remark}

A locally finite  \qt\ graph $G$ is \emph{edge-accessible} if there exists a positive integer $k$ such that any two ends of~$G$ can be separated by at most~$k$ edges and it is vertex-accessible if there exists a positive integer~$k'$ such that any two ends of~$G$ can be separated by at most~$k'$ vertices.

The following result generalizes a graph theoretical characterisation of accessibility of finitely generated groups: Thomassen and Woess~\cite[Theorem 1.1]{ThomassenWoess} proved that a finitely generated group is accessible if and only if it has an edge-accessible locally finite Cayley graph.

\begin{thm}\label{thm_accessmain}
Let $G$ be a connected locally finite \qt\ graph.
Then the following statements are equivalent.
\begin{enumerate}[label=\rm (\arabic*)]
\item\label{itm_accessmainAcc} $G$ is accessible.
\item\label{itm_accessmainEAcc} $G$ is edge-accessible.
\item\label{itm_accessmainVAcc} $G$ is vertex-accessible.
\end{enumerate}
\end{thm}

Before we prove Theorem~\ref{thm_accessmain}, we need another result.
Recall that a \td\ efficiently distinguishes two ends if there is an adhesion set $V_{t_1} \cap V_{t_2}$ separating them such that no set of smaller size than $V_{t_1} \cap V_{t_2}$ separates them.

\begin{thm}\label{thm_access}
	Let~$G$ be an vertex-accessible connected locally finite graph and let $\Gamma$ be a group acting \qt ly on~$G$.
	Then there exists a $\Gamma$-invariant \td\ $(T,\Vcal)$ of~$G$ of finite adhesion such that $(T,\Vcal)$ distinguishes all ends of~$G$ efficiently and such that there are only finitely many $\Gamma$-orbits on~$E(T)$. 
\end{thm}

\begin{proof}
Since $G$ is vertex-accessible, there exists $k\in\N$ such that every two ends can be separated by at most~$k$ vertices.
By Theorem~\ref{thm_canonicalTD} we find a $\Gamma$-invariant \td\ $(T,\Vcal)$ of~$G$ of adhesion at most~$k$ that distinguishes all ends efficiently.

For every adhesion set $V_t\cap V_{t'}$ that does not separate any two ends efficiently, we contract the edge $tt'$ in~$T$ and assign the vertex set $V_t\cup V_{t'}$ to the new node.
It is easy to check that the resulting pair $(T',\Vcal')$ is again a \td.
It only has adhesion sets that distinguish ends efficiently.
Note that $\Gamma$ still acts on $(T',\Vcal')$ as the set of adhesion sets that do not separate ends efficiently is $\Gamma$-invariant.
A result of Thomassen and Woess~\cite[Proposition 4.2]{ThomassenWoess} says that there are only finitely many vertex sets $S$ of size at most~$k$ containing a fixed vertex such that for two components $C_1,C_2$ of $G-S$ every vertex of~$S$ has a neighbour in~$C_1$ and in~$C_2$.
It follows that there are only finitely many orbits of adhesion sets that separate ends efficiently.
This proves the assertion.
\end{proof}

\begin{proof}[Proof of Theorem~\ref{thm_accessmain}.]
The equivalence of~\ref{itm_accessmainEAcc} and~\ref{itm_accessmainVAcc} is trivial.

To prove that \ref{itm_accessmainVAcc} implies \ref{itm_accessmainAcc}, let $G$ be vertex-accessible and let $\Gamma$ be a group acting on~$G$ with only finitely many orbits.
By Theorem~\ref{thm_access} we find a $\Gamma$-invariant \td\ $(T,\Vcal)$ of~$G$ of finite adhesion such that $(T,\Vcal)$ distinguishes all ends of~$G$ efficiently and such that there are only finitely many $\Gamma$-orbits on~$E(T)$. 
By Proposition~\ref{connected TD}, we may assume that all adhesion sets are connected.
We prove the assertion by induction on the number of $\Gamma$-orbits of adhesion sets of $(T,\Vcal)$.
Let $tt'\in E(T)$.
For every edge $t_1t_2\in E(T)$ that does not lie in the same $\Gamma$-orbits as~$tt'$, we contract the edge $t_1t_2$ in~$T$ and assign the vertex set $V_{t_1}\cup V_{t_2}$ to the new node.
Let $T'$ be the resulting tree and $\Vcal'=\{V_s\mid s\in V(T')\}$.
It is easy to verify that $(T',\Vcal')$ is a \td.
The only edges of~$T'$ are those that have their origin in the $\Gamma$-orbit of the edge $tt'\in E(T)$ and $\Gamma$ still acts on $(T',\Vcal')$ such that $(T',\Vcal')$ is a connected basic \td\ of~$G$ with only connected adhesion sets.
Lemma~\ref{TD_implies_TA} implies that $G$ is a non-trivial \ta\ $G_1\ast_{T'}G_2$ with respect to group actions, where the graphs $G_1$ and $G_2$ are induced by the parts of~$(T',\Vcal')$.
The \td\ $(T,\Vcal)$ induces a \td\ $(T_W,\Wcal)$ on the parts $W$ of~$(T',\Vcal')$ and there are less $\Gamma_W$-orbits on the adhesion sets of $(T_W,\Wcal)$ than $\Gamma$-orbits on the adhesion sets of $(T,\Vcal)$.
Thus, we can apply induction.
This shows~\ref{itm_accessmainAcc}.

To prove that~\ref{itm_accessmainAcc} implies~\ref{itm_accessmainVAcc}, we will use the graph classes $\Gcal_i$ and $\Gcal$ as defined in~\ref{rem_access3} and show inductively that every $\Gcal_i$ contains only vertx-accessible connected \qt\ locally finite graphs.
This is obviously true for $\Gcal_0$.
Let $G\in \Gcal_i$ for $i>0$.
Then there are $G_1,G_2\in\bigcup_{j<i}\Gcal_j$ such that $G$ is a \ta\ $G_1\ast_T G_2$ of finite adhesion.
By induction, we may assume that $G_1$ and $G_2$ are edge-accessible and \qt.
Note that quasi-transitivity of~$G$ follows from Lemma~\ref{lem_TAareQT} since $G_1$ and~$G_2$ are \qt.
For $i=1,2$, let $k_i$ be a positive number such that any two ends of~$G_i$ can be separated by at most $k_i$ many vertices.
Let $(T,\Vcal)$ be the \td\ we obtain from the \ta\ $G_1\ast_T G_2$ according to Remark~\ref{TA_implies_TD}.
Let $k$ be the maximum of $k_1$, $k_2$ and the size of adhesion sets of $G_1\ast_T G_2$.

Let $Q,R$ be two rays in different ends $\omega_Q$, $\omega_R$ of~$G$, respectively.
If there is some adhesion set $V_t\cap V_{t'}$ such that $Q$ and~$R$ have tails that are separated by~$V_t\cap V_{t'}$, then the ends they lie in must be separated by that adhesion set as well.
Hence, they are separable by a separator of order at most~$k$.
So we may assume that, eventually, they lie on the same side of each separator.
By Proposition~\ref{capture thick ends_new}\;\ref{itm_cte_1} every end of~$G$ is captured either by an end or by a node of~$T$.
Thus and since no separator separates any tails of~$Q$ and~$R$, their ends are captured by the same node or end of~$T$.
By Proposition~\ref{capture thick ends_new}\;\ref{itm_cte_3} an end of~$T$ captures a unique end of~$G$.
Thus, $\omega_Q$ and $\omega_R$ are captured by the same node of~$T$.
By Proposition~\ref{capture thick ends_new}\;\ref{itm_cte_4} every end of~$G$ that is captured by a node $t\in V(T)$ corresponds to a uniquely determined end of~$G[V_t]$.
These ends can be separated by a separator~$S$ in~$G[V_t]$ of order at most~$k$ by assumption.
However, $S$ need not be a separator of~$G$ that separates those ends.
Still, it is possible to enlarge~$S$ to a separator of~$G$ that separates $\omega_Q$ and $\omega_R$ and still has bounded size:
every vertex of~$S$ has distance at most~$K$, the maximum diameter of the adhesion sets measured in~$G_1$ and in~$G_2$ to only finitely many adhesion sets that are contained in~$V_t$ as $G$ is locally finite; so we can add all these adhesion sets to~$S$ and obtain a set $S'$.
As $G$ is quasi-transitive, the size of~$S'$ only depends on~$k$, the number of orbits of vertices of~$G$, the maximum number of adhesion sets in~$V_t$ that have distance at most~$K$ to a common vertex and the size of any adhesion set of $(T,\Vcal)$, in particular, it is bounded by some $\ell\in\N$ and it is independent of the chosen ends.
If we show that $S'$ separates $\omega_Q$ and $\omega_R$, then it follows immediately that $G$ is edge-accessible.

Let $P=\ldots, x_{-1},x_0,x_1,\ldots$ be a double ray with its tail $x_0,x_1,\ldots$ in~$\omega_Q$ and its tail $x_0,x_{-1},\ldots$ in~$\omega_R$.
Since both ends $\omega_Q$ and $\omega_R$ are captured by~$V_t$, there are infinitely many $x_i$ with $i>0$ that lie in~$V_t$ and infinitely many $x_i$ with $i<0$ that lies in~$V_t$.
Let us assume $x_0\in V_t$.
Whenever the ray $P^+:=x_0x_1\ldots$ leaves $V_t$ through an adhesion set $V_t\cap V_{t'}$, it must reenter $V_t$ and this must happen through the same adhesion set.
Since $S$ is finite and separates $\omega_Q$ and $\omega_R$, there are $i_1,i_2\in\Z$ such that no $x_i\in V_t$ with $i\geq i_1$ is separated in $G[V_t]$ by~$S$ from~$\omega_Q$ and no $x_i\in V_t$ with $i\leq i_2$ is separated in $G[V_t]$ by~$S$ from $\omega_R$.
Then there must be some path $x_i,\ldots,x_j$ with $j\geq i+1$ and whose inner vertices lie outside of~$V_t$ such that $x_j$ is not separated by~$S$ from~$\omega_Q$ and $x_i$ is not separated by~$S$ from~$\omega_R$.
Thus, the shortest $x_i$-$x_j$ path in $G[V_t]$ meets~$S$.
As $x_i$ and~$x_j$ lie in a common adhesion set, we conclude that this lies in~$S'$.
Thus, $S'$ separates $\omega_Q$ from~$\omega_R$ in~$G$.
This shows~\ref{itm_accessmainVAcc}.
\end{proof}

In the proof of the implication \ref{itm_accessmainVAcc} to \ref{itm_accessmainAcc} of Theorem~\ref{thm_accessmain} we chose a specific way to split the factors.
(It was based on a $\Gamma$-invariant \td\ of~$G$.)
We do not know if we can split arbitrary in each step and still have to end in a terminal factorisation.
But we conjecture that this is true.

\begin{conj}
Let $G$ be an accessible connected \qt\ locally finite graph.
Every process of splittings must end after finitely many steps.
\end{conj}

Accessibility of finitely generated groups received lots of attention after Wall~\cite{Wall-AccessibilityConjecture} conjectured that all finitely generated groups are accessible and among the main results in this area are Dunwoody's results that Wall's conjecture is false in general~\cite{D-AnInaccessibleGroup} but true for (almost) finitely presented groups~\cite{dunaccess}.
In the case of \qt\ locally finite graphs, the investigation focused on edge-accessible graphs, see \cite{accessmatthias,Moeller-Accessibility,ThomassenWoess}.
However, Theorem~\ref{thm_accessmain} enables us to carry over these results to accessible graphs.


\section{Applications}
\label{Sec_App}

\subsection{Stallings' theorem}\label{Sec_Stallings}

In this section we will discuss how to obtain Stallings' theorem from our results.

Let $\Gamma$ be a finitely generated group with infinitely many ends and let $G$ be a locally finite Cayley graph of~$\Gamma$.
Then $G$ has infinitely many ends, too.
By Theorem~\ref{stallingsgraph}, $G$ is a non-trivial \ta\ $G_1\ast_T G_2$ of finite adhesion.
Since it has finite adhesion and $\Gamma$ acts regularly\footnote{i.\,e.\ for every two $u,v\in V(G)$ there is a unique element of~$\Gamma$ mapping $u$ to~$v$} on~$G$, the stabiliser in~$\Gamma$ of an edge of~$T$, which is just the stabiliser in~$\Gamma$ of the corresponding adhesion set, is finite.
Hence, Bass-Serre theory leads to Stallings' theorem.

\begin{thm}{\rm \cite{Stallings}}
If a finitely generated group has more than one end, then it is either a free product with amalgamation over a finite subgroup or an HNN-extension over a finite subgroup.\qed
\end{thm}

Note that \ta s of Type 1 respecting the actions of groups acting on the factors lead to free products with amalgamation and \ta s of a graph with itself with respect to the action of a group leads to an HNN-extension.

\subsection{Graphs without thick ends}\label{Sec_thin}

Let us apply our main results to connected \qt\ locally finite graphs that have only thin ends.
First, we want to see that such graphs are edge-accessible, so we can apply Theorem~\ref{thm_access} and look at terminal factorisations of them.
But before we go into the proof, we need some definitions.

Let $G$ and $H$ be graphs.
A map $\varphi\colon V(G)\to V(H)$ is a \emph{$(\gamma,c)$-quasi-isometry} if there are constants $\gamma\geq 1,c\geq 0$ such that
\[
\gamma\inv d_G(x,y)-c\leq d_H(\varphi(x),\varphi(y))\leq\gamma d_G(x,y)+c
\]
for all $x,y\in V(G)$ and such that $\sup\{d_H(x,\varphi(V(G)))\mid x\in V(H)\}\leq c$.
We then say that $G$ is \emph{quasi-isometric} to~$H$.

Kr\"on and M\"oller \cite[Theorem~5.5]{kron2008quasi} showed that a connected \qt\ locally finite graph has only thin ends if and only if it is quasi-isometric to a tree.
Trees are obviously edge-accessible and it follows from the definition of edge-accessibility that the class of edge-accessible \qt\ locally finite graphs is invariant under quasi-isometries.
Thus, we have verified the following.

\begin{prop}\label{accessthin}
	Every connected locally finite \qt\ graph that has only thins ends is edge-accessible.\qed
\end{prop}

We mention that Thomassen and Woess~\cite[Theorem 5.3]{ThomassenWoess} showed Proposition~\ref{accessthin} for transitive graphs directly with a nice graph theoretical argument.
It is not too hard to modify their argument in such a way that the proof works for \qt\ graphs as well.

Another result we need for our investigation here is due to Thomassen.

\begin{prop}{\rm \cite[Proposition 5.6.]{hadwiger}}
	\label{transitive thick ends}
	If~$G$ is an infinite connected \qt\ locally finite graph  with only one end, then the end is thick.\qed
\end{prop}

Recently, Carmesin et al.~\cite[Theorem~5.1]{TDOneEnded} extended Proposition~\ref{transitive thick ends} to graphs that need not be locally finite.

Now we are able to give a new characterisation of connected \qt\ locally finite graphs with only thin ends.

\begin{thm}\label{thin ends main}
	A connected \qt\ locally finite graph has only thin ends if and only if it has a terminal factorisation of only finite graphs.	
\end{thm}

\begin{proof}
	Let $G$ be a connected \qt\ locally finite graph.
	First, let us assume that every end of~$G$ is thin.
	By Theorem~\ref{accessthin}, $G$ is edge-accessible.
	So Theorem~\ref{thm_access} implies that $G$ is accessible and hence has a terminal factorisation.
	All the factors of that terminal factorisation have at most one end.
	Since they are quasi-transitive by Proposition~\ref{quasitransitivebags}, they cannot have one end due to Proposition~\ref{transitive thick ends}.
	So they are locally finite graphs without ends, which implies that they are finite graphs.
	
	For the other direction, we follow the steps to factorise $G$, factorise each of its factors and so on until we end up with a terminal factorisation.
	Note that by Proposition~\ref{capture thick ends_new}\,\ref{itm_cte_2} every thick end of~$G$ is captured by nodes of the involved basic \td s.
	So if $G$ had a thick end, then one of the factors of the terminal factorisation must have a thick end, which is impossible as these factors are finite by assumption.
	Thus, all ends of~$G$ are thin.
\end{proof}	

Note that there are several characterisations of (\qt\ or Cayley) graphs that are quasi-isometric to trees, see e.\,g.\ Antol\'in~\cite{virtuallyfreegroup} and Kr\"on and M\"oller~\cite{kron2008quasi}.
We enlarged their list of characterisations by our theorem.

A natural class of \qt\ graphs are Cayley graphs.
So our theorems apply in particular for such graphs and we obtain as a corollary of Theorem~\ref{thin ends main} a result for virtually free groups.
A group~$\G$ is \emph{virtually free} if it contains a free subgroup of finite index.

Woess~\cite{tree-like} showed that a finitely generated group is virtually free if and only if every end of any of its locally finite Cayley graphs is thin.
Thus we directly obtain the following corollary.

\begin{coro}
	A finitely generated group is virtually free if and only if any of its locally finite Cayley graphs has a terminal factorisation of only finite graphs.\qed
\end{coro} 

In~\cite{H-QuasiIsometriesAndTA} the interplay between \ta s and quasi-isometries is investigated further and the results of this section are extended to graphs other than trees in two ways.
First, it is shown that the quasi-isometry type of (iterated) \ta s only depend on the quasi-isometry types of the infinite factors.
Then, in the case of accessible infinitely-ended graphs, it is shown that the quasi-isometry types of the graphs determine the quasi-isometry types of the infinite factors in any of its terminal factorisations.

\subsection{Planar graphs}\label{Sec_planar}

Mohar, see~\cite{Mohar06}, raised the question whether \ta s are powerful enough to characterise planar transitive locally finite graphs in terms of finite or one-ended locally finite planar transitive graphs.
The aim of this section is to answer his question in the affirmative in case of planar \qt\ graphs.

Dunwoody~\cite{D-PlanarGraphsAndCovers} proved that planar \qt\ locally finite graphs are edge-accessible, see also~\cite{H-PlanarTransitivity}.
This allows us to apply Theorem~\ref{thm_accessmain} to these graphs.
We directly obtain the following result.

\begin{theorem}
For every planar connected \qt\ locally finite graph $G$ there are finitely many planar connected \qt\ locally finite graphs $G_1,\ldots,G_n$ with at most one end such that $G$ can be obtained by finitely many (iterated) \ta s of $G_1,\ldots,G_n$.\qed
\end{theorem}

\subsection{Hyperbolic graphs}\label{Sec_hyperbolic}

For our last application, we look at hyperbolic graphs.
The aim is to give a characterisation of \qt\ locally finite hyperbolic graphs in terms of their terminal factorisations.
But before we state the main theorem of this section, we need some definitions and preliminary results.

Let $\delta\geq 0$ and let $G$ be a graph.
Then $G$ is \emph{$\delta$-hyperbolic} if for all vertices $x_1,x_2,x_3\in V(G)$ and all geodesics $P_{i,j}$ between $x_i$ and $x_j$, every vertex of the path $P_{1,2}$ has distance at most $\delta$ to a vertex on $P_{2,3}\cup P_{1,3}$.
We call $G$ \emph{hyperbolic} if it is $\delta$-hyperbolic for some $\delta\geq 0$.

Let $\gamma\geq 1$ and $c\geq 0$.
A finite walk is called a \emph{$(\gamma,c)$-\qg} if it is the image of a $(\gamma,c)$-quasi-isometry of some geodesic with the same end vertices.

\begin{lemma}\label{lem_geodQuasigeod}
Let $G_1, G_2$ be connected locally finite graphs and let $G=G_1\ast G_2$ be a \ta\ of~$G_1$ and~$G_2$ such that the adhesion sets in~$G_1$ have bounded diameter in~$G_1$.
Then there are some $\gamma\geq 1$, $c\geq 0$ such that every geodesic in~$G_1$ is a $(\gamma,c)$-\qg\ in~$G$.
\end{lemma}

\begin{proof}
Let $\gamma'$ be the maximum distance in~$G_1$ between vertices of an adhesion set.
If $\gamma'=1$, then it is straightforward to see that every geodesic in~$G_1$ is a geodesic in~$G$.
So we may assume $\gamma'\geq 2$.
Let $P_1$ be a geodesic in~$G_1$ and let $P$ be a geodesic in~$G$ with the same end vertices as~$P_1$.
Whenever $P$ contains a vertex outside of~$G_1$, it must have left $G_1$ through an adhesion set and reentered through the same adhesion set.
(Note that we consider the \td\ defined by the \ta\ $G_1\ast G_2$ as in Remark~\ref{TA_implies_TD}.)
We replace every maximal subpath of~$P$ all whose inner vertices lie outside of~$G_1$ by a path in~$G_1$ with the same end vertices.
As these end vertices lie in a common adhesion set, the length of the replacement path is at most~$\gamma'$ and it replaces a path of length at least~$2$.
We end up with a walk in~$G_1$ that has the same end vertices as~$P$ and that is at most $\gamma'/2$ times as long as~$P$.
Since $P_1$ is shorter than that walk, its length is at most $\gamma'/2$ times as long as~$P$.
As the same applies for every subpath of~$P_1$, we conclude that $P_1$ is a $(\gamma'/2,0)$-\qg\ in~$G$.
\end{proof}

The reason that we are interested in \qg s is because they still lie relatively close to geodesics in hyperbolic graphs as the following lemma shows.

\begin{lemma}{\rm \cite[Th\'eor\`eme 3.1.4]{CDP}}\label{lem_QuasigeodHyper}
Let $G$ be a locally finite $\delta$-hyperbolic graph. For all $\gamma\geq 1$ and $c\geq 0$, there is a constant $\kappa = \kappa(\delta,\gamma,c)$ such that for every two vertices $x,y$ of~$G$ every $(\gamma,c)$-quasi-geodesic between them lies in a $\kappa$-neigh\-bour\-hood around every geodesic between $x$ and $y$ and vice versa.\qed
\end{lemma}

\begin{lemma}\label{lem_GeodGVt}
Let $(T,\Vcal)$ be a connected basic \td\ of a connected locally finite graph~$G$ such that every part induces a hyperbolic graph.
Let $x,y\in V_t$ for some $t\in V(T)$.
Then there exists $\lambda=\lambda(\delta,\gamma)$ such that every geodesic in~$G$ between $x$ and~$y$ lies in a $\lambda$-neighbourhood of every geodesic in $G[V_t]$ between $x$ and~$y$ and vice versa, where $\gamma$ is the diameter of any adhesion set.
\end{lemma}

\begin{proof}
Let $P, P'$ be a geodesic in~$G$, in~$G[V_t]$ between $x$ and~$y$, respectively.
We modify $P$ the same way we did it in the proof of Lemma~\ref{lem_geodQuasigeod}:
we replace every maximal subpath of~$P$ all whose inner vertices lie outside of~$G[V_t]$ by a shortest path in~$G[V_t]$ with the same end vertices.
This is possible as whenever $P$ contains a vertex outside of~$G[V_t]$ it must have left $G[V_t]$ through an adhesion set and reentered through the same adhesion set.
Let $P''$ be the walk obtained from~$P$ after these replacements.
If $\gamma=1$, then $P=P''$ and if $\gamma\neq 1$, then the length of~$P''$ is at most $\gamma/2$ times the length of~$P$ as paths of length at least~$2$ got replaced by paths of length at most~$\gamma$.
Set $\gamma':=\max\{1,\gamma/2\}$.
Now let $a,b$ be vertices of~$P''$.
Then there are vertices $a',b'$ on~$P\cap P''$ of distance at most $\gamma/2$ to~$a$, to~$b$, respectively.
Hence, we have
\[
d_{P''}(a,b)\leq \gamma+d_{P''}(a',b')\leq \gamma+\gamma'd_G(a',b')\leq (1+\gamma')\gamma+\gamma'd_G(a,b).
\]
So $P''$ is $(\gamma',(1+\gamma')\gamma)$-\qg.
Applying Lemma~\ref{lem_QuasigeodHyper}, we find $\kappa$ depending only on $\delta$ and~$\gamma$ such that $P''$ lies in a $\kappa$-neighbourhood of~$P'$ and vice versa.
Since $P$ lies in a $\gamma/2$-neighbourhood of~$P''$ and vice versa, we have shown the assertion.
\end{proof}

The following theorem is our main result for \qt\ locally finite hyperbolic graphs: it shows that \ta s behave well with respect to hyperbolicity.

\begin{thm}\label{thm_hyperb}
Let $G_1$ and $G_2$ be connected locally finite graphs and let $G$ be a \ta\ $G_1\ast G_2$ such that the adhesion sets in ~$G_i$ have bounded diameter in~$G_i$ for $i=1,2$.
Then $G$ is hyperbolic if and only if $G_1$ and~$G_2$ are hyperbolic.
\end{thm}

\begin{proof}
First, let $G$ be $\delta$-hyperbolic.
Let $x_1,x_2,x_3\in V(G_1)$ and let $P_{ij}$ be a geodesic in~$G_1$ and $P'_{ij}$ be a geodesic in~$G$ between $x_i$ and~$x_j$ for all $i\neq j$.
Let $x\in P_{12}$.
By Lemma~\ref{lem_geodQuasigeod} and its proof, $P_{12}$ is $(\gamma,0)$-\qg\ for $\gamma=\max\{1,\beta/2\}$, where $\beta$ is	 the maximum distance in~$G_1$ between two vertices in a common adhesion set in~$G_1$.
By Lemma~\ref{lem_QuasigeodHyper}, there is some $x'\in P'_{12}$ of distance at most~$\kappa$ to~$x$ for some $\kappa\geq 0$.
Since $G$ is $\delta$-hyperbolic, we find $y'\in P'_{13}\cup P'_{23}$ of distance at most~$\delta$ to~$x'$.
Again by Lemma~\ref{lem_QuasigeodHyper}, we find $y\in P_{13}\cup P_{23}$ of distance at most~$\kappa$ to~$y'$.
Hence, we have $d_G(x,y)\leq 2\kappa+\delta$.
Lemma~\ref{lem_geodQuasigeod} implies $d_{G_1}(x,y)\leq \gamma(2\kappa+\delta)$.
Thus, $G_1$ is $\gamma(2\kappa+\delta)$-hyperbolic.
Analogously, $G_2$ is $\gamma(2\kappa+\delta)$-hyperbolic.

Now let $G_1$ and~$G_2$ be hyperbolic.
Then there is some $\delta\geq 0$ such that $G_1$ and~$G_2$ are $\delta$-hyperbolic.
Let $\gamma_i$ be the maximum distance between vertices in a common adhesion set in~$G_i$ for $i=1,2$ and let $\gamma:=\max\{\gamma_1,\gamma_2\}$.
We consider the canonical \td\ $(T,\Vcal)$ as discussed in Remark~\ref{TA_implies_TD}.
Note that the parts of $(T,\Vcal)$ induce graphs that are isomorphic to either $G_1$ or~$G_2$, so they are hyperbolic.
Let $x_1,x_2,x_3\in V(G)$ and let $P_{ij}$ be a geodesic between $x_i$ and~$x_j$.
Let $t_1,t_2\in V(T)$ of minimum distance to each other with $x_i\in V_{t_i}$ for $i=1,2$ and let $T_{12}$ be the $t_1$-$t_2$ path in~$T$.
Note that every node on~$T_{12}$ and every adhesion set $V_t\cap V_{t'}$ with $tt'\in E(T_{12})$ contains a vertex of~$P_{12}$ and also of $P_{13}\cup P_{23}$.

Let $x\in P_{12}$.
Let $t'\in V(T)$ closest to~$T_{12}$ with $x\in V_{t'}$ and let $t\in T_{12}$ closest to~$t'$.
We say \emph{$P_{12}$ passes through $V_t$ in parallel to~$P_{13}$} either if $t=t_1$ and $P_{13}$ contains a vertex of the adhesion set $V_t\cap V_{t'_2}$, where $t'_2$ is the neighbour of~$t$ on~$T_{12}$, or if $t$ is neither $t_1$ nor $t_2$ and both $V_t\cap V_{t'_1}$ and $V_t\cap V_{t'_2}$ contain vertices of~$P_{13}$, where $t'_i$ is the neighbour of~$t$ on~$T_{12}$ closest to~$t_i$ for $i=1,2$.
Analogously, \emph{$P_{12}$ passes through $V_t$ in parallel to~$P_{23}$} either if $t=t_2$ and $P_{23}$ contains a vertex of the adhesion set $V_t\cap V_{t'_1}$, where $t'_1$ is the neighbour of~$t$ on~$T_{12}$, or if $t$ is neither $t_1$ nor $t_2$ and both $V_t\cap V_{t'_1}$ and $V_t\cap V_{t'_2}$ contain vertices of~$P_{23}$, where $t'_i$ is the neighbour of~$t$ on~$T_{12}$ closest to~$t_i$ for $i=1,2$.

First, let us assume that $P_{12}$ passes through $V_t$ in parallel to~$P_{13}$.
If $t=t_1$, let $u_1:=v_1:=x_1$, let $u_2$ be the last vertex on~$P_{12}$ in $V_t\cap V_{t'_2}$, and let $v_2$ be a vertex on~$P_{13}$ in $V_t\cap V_{t'_2}$.
If $t\neq t_1$, let $u_1$ be the first vertex on~$P_{12}$ in $V_t\cap V_{t'_1}$ and $u_2$ be the last vertex on~$P_{12}$ in $V_t\cap V_{t'_2}$ and let $v_1, v_2$ be on~$P_{13}$ in $V_t\cap V_{t'_1}$, in $V_t\cap V_{t'_2}$, respectively.
Note that by the choice of $u_1$ and~$u_2$, the vertex~$x$ lies between $u_1$ and~$u_2$ on~$P_{12}$.
Let $P$ be a geodesic in~$V_t$ between $u_1$ and~$u_2$ and let $P'$ be a geodesic in~$V_t$ between $v_1$ and~$v_2$.
Let $Q_i$ be a geodesic in~$V_t$ between $u_i$ and~$v_i$ for $i=1,2$.
Looking at a $v_1,v_2,u_2$, we conclude by $\delta$-hyperbolicity that any geodesic between $v_1$ and~$v_2$ lies in a $\delta$-neighbourhood of $Q_2\cup P'$, so $P$ lies in a $2\delta$-neighbourhood of $Q_1\cup P'\cup Q_2$.
As the lengths of~$Q_1$ and of~$Q_2$ are bounded by~$\gamma$, we conclude that $P$ lies in a $(2\delta+\gamma)$-neighbourhood of~$P'$.
Lemma~\ref{lem_GeodGVt} implies the existence of some $\lambda$ such that $P$ contains a vertex $y_1$ of distance at most $\lambda$ from~$x$.
We just showed that $P'$ contains a vertex~$y_2$ of distance at most~$2\delta+\gamma$ from~$y_1$ and Lemma~\ref{lem_GeodGVt} shows the existence of a vertex $y_3$ on~$P_{23}$ with $d(x,y_3)\leq 2\lambda+2\delta+\gamma$.

Analogously, we conclude in the case that $P_{12}$ passes through $V_t$ in parallel to~$P_{23}$ that $P_{23}$ contains a vertex of distance at most $d(x,y_3)\leq 2\lambda+2\delta+\gamma$ from~$x$.

Let us now assume that $P_{12}$ passes through $V_t$ neither in parallel to~$P_{13}$ nor in parallel to~$P_{23}$.
If $t=t_1$, let $u_1:=v_1:=x_1$.
If $t\neq t_1$, let $t'_1$ be the neighbour of~$t$ on~$T_{12}$ closest to~$t_1$ and let $u_1$ be the first vertex on~$P_{12}$ in $V_t\cap V_{t'_1}$ and let $v_1$ be a vertex on~$P_{13}$ in $V_t\cap V_{t'_1}$.
If $t=t_2$, let $u_2:=w_2:=x_1$.
If $t\neq t_2$, let $t'_2$ be the neighbour of~$t$ on~$T_{12}$ closest to~$t_2$ and let $u_2$ be the last vertex on~$P_{12}$ in $V_t\cap V_{t'_2}$ and let $w_2$ be a vertex on~$P_{23}$ in $V_t\cap V_{t'_2}$.
Let $t_3\in V(T)$ of minimum distance to~$t$ such that $x_3\in V_{t_3}$.
If $t=t_3$, let $w_1:=v_2:=x_3$.
If $t\neq t_3$, let $t'_3$ be the neighbour of~$t$ in~$T$ closest to~$t_3$ and let $w_1$ be a vertex on~$P_{13}$ in $V_t\cap V_{t'_3}$ and let $v_2$ be a vertex on~$P_{23}$ in $V_t\cap V_{t'_3}$.
(Note that both $P_{13}$ and $P_{23}$ must pass through the adhesion set $V_t\cap V_{t'_3}$according to the definition of a \td.)

We consider a couple of geodesics in~$G[V_t]$: let $P_u$, $P_v$, $P_w$ be a geodesic in $G[V_t]$ between $u_1,u_2$, between $v_1,v_2$, between $w_1,w_2$, respectively and let $P_{uv}$, $P_{vw}$, $P_{uw}$ be a geodesic in $G[V_t]$ between $u_1$ and~$v_1$, between $v_2$ and~$w_1$, between $u_2$ and~$w_2$, respectively.
Similar to the case that $P_{12}$ passes through $V_t$ in parallel to~$P_{13}$ we conclude that $P_u$ lies in a $4\delta$-neighboudhood of $P_{uv}\cup P_v\cup P_{vw}\cup P_w\cup P_{uw}$ and hence in a $(4\delta+\gamma)$-neighbourhood of $P_v\cup P_w$.

Let $\lambda$ be the value obtained in Lemma~\ref{lem_GeodGVt}.
Then there is a vertex $y_1$ on~$P_u$ of distance at most $\lambda$ from~$x$.
As we just showed, we find $y_2$ on either $P_u$ or~$P_w$ with $d(y_1,y_2)\leq 4\delta+\gamma$ and Lemma~\ref{lem_GeodGVt} then implies the existence of a vertex $y_3$ on either $P_{13}$ or $P_{23}$ with $d(y_2,y_3)\leq\lambda$.
So we have $d(x,y_3)\leq 2\lambda+4\delta+\gamma$.
This proves that $G$ is $(2\lambda+4\delta+\gamma)$-hyperbolic.
\end{proof}

As a corollary of Theorem~\ref{thm_hyperb}, we obtain a characterisation of \qt\ locally finite hyperbolic graphs in terms of their terminal factorisations.

\begin{coro}
A connected \qt\ locally finite graph is hyperbolic if and only if it admits a terminal factorisation such that all its factors are connected \qt\ locally finite hyperbolic graphs with at most one end.
\end{coro}

\begin{proof}
Let $G$ be a connected \qt\ locally finite graph.
If $G$ is one-ended, then it is a terminal factorisation of itself and the assertion holds trivially.
So let us assume that $G$ has more than one end.

First, let us assume that $G$ is hyperbolic.
By~\cite[Theorem~4.3]{accessmatthias}, it is an edge-accessible graph.
Thus it is accessible and has a terminal factorisation by Theorem~\ref{thm_accessmain}.
So there are connected \qt\ locally finite graphs $G_1,\ldots, G_n$, $H_1,\ldots,H_{n-1}$ with $G=H_{n-1}$ such that each $G_i$ has at most one end and for every $i\leq n-1$, the graph $H_i$ is a \ta\ $H\ast H'$ of finite adhesion, where
\[
H,H'\in\{G_j\mid 1\leq j\leq n\}\cup\{H_j\mid 1\leq j<i\}.
\]
(We may assume that all $G_i$ are indeed needed at some point during these \ta s.)
By repeated application of Lemma~\ref{thm_hyperb}, each $H_i$, and thus each $G_i$ is hyperbolic.

Conversely, if $G$ has a terminal factorisation into connected finite or connected \qt\ locally finite hyperbolic one-ended graphs, then each of the previous factors we considered for obtaining the terminal factorisation are hyperbolic by Lemma~\ref{thm_hyperb}.
In particular, $G$ is hyperbolic.
\end{proof}

\bibliographystyle{plain}
\bibliography{collective}

\end{document}